\documentclass{article}
\usepackage{amsmath,amssymb,amsthm}
\usepackage{geometry}
\usepackage{booktabs}
\usepackage{hyperref}

\theoremstyle{definition}
\newtheorem{definition}{Definition}[section]
\newtheorem{condition}{Condition}[section]
\newtheorem{proposition}{Proposition}[section]
\newtheorem{lemma}{Lemma}[section]
\newtheorem{theorem}{Theorem}[section]

\newcommand{\hparam}{\hbar}    
\newcommand{\hscale}{h}        
\newcommand{\norm}[1]{\| #1 \|}
\usepackage{enumitem}

\hypersetup{
    colorlinks=true,
    linkcolor=blue,
    citecolor=green,
    urlcolor=cyan
}

\begin{document}

\title{From Weak Nonlinear Perturbation to the Homotopy Analysis Method: A Rigorous Derivation and Theoretical Unification}
\author{Hang Xu \\
        State Key Lab of Ocean Engineering, School of Ocean and Civil Engineering \\
        Shanghai Jiao Tong University, Shanghai 200240, China \\
        \texttt{hangxu@sjtu.edu.cn}
}
\date{\today}  
\maketitle

\noindent \textbf{Abstract:}
The Homotopy Analysis Method (HAM) is a widely used analytical approach for nonlinear problems, but its theoretical foundation lacks clear rigor. Its correlation with perturbation theory remains ambiguous, causing confusion in existing literature. This study shows that the method’s basic homotopy deformation equation can be naturally derived from weak-nonlinearity perturbation theory. It constructs a specific expression and optimizes core parameters (optimal auxiliary linear operator, convergence-control parameter, auxiliary function) to mitigate the inherent strong nonlinearity of the nonlinear operator.
Extending perturbation theory’s small parameter $\varepsilon$ to [0,1] enables systematic homotopy deformation, connecting the linear auxiliary system (at $\varepsilon=0$) with the original nonlinear problem (at $\varepsilon=1$) and confirming this method as a structured, adaptive generalization of classical perturbation theory.
Furthermore, this study rigorously proves that the Homotopy Perturbation Method (HPM) is a special case of the above-mentioned approach. It can be directly derived by fixing the optimal auxiliary linear operator as the nonlinear system’s linear component, and setting the convergence-control parameter and auxiliary function to specific values, making it a degenerate form of the original method.
This study clarifies the perturbation-theoretic origin of the method, defines the hierarchical subordination of perturbation methods to it, unifies the theoretical framework of homotopy-based nonlinear analytical methods, rectifies common misconceptions in existing literature, and provides guidance for their rational application, comparative analysis and further development.

\vspace{2mm}
\noindent \textbf{Keywords}: {Homotopy Analysis Method; Perturbation Theory; Homotopy Perturbation Method; Weak Nonlinearity; Convergence Control}

\section{Introduction}

Nonlinear phenomena are pervasive in applied mathematics and engineering science, presenting core challenges for understanding natural phenomena and optimizing engineering systems \cite{Nayfeh1981}. Traditional small-parameter perturbation methods linearize nonlinear terms through intrinsic small parameters and have been extensively applied to weakly nonlinear problems including small-amplitude vibrations and low-Reynolds-number flows \cite{Nayfeh2000}. However, their effectiveness diminishes for strongly nonlinear problems lacking obvious small parameters. 
To overcome this limitation, researchers have developed extended perturbation methods, with singular perturbation and multi-scale perturbation methods being particularly significant \cite{Bender1999, Kevorkian1996}. Singular perturbation methods address problems where small parameters multiply highest-order derivative terms, leading to boundary layer effects that traditional methods cannot capture \cite{Bender1999}. The core approach involves constructing separate outer and inner solutions matched through asymptotic principles to form uniformly valid approximations. Prandtl's application of singular perturbation theory to establish boundary layer equations exemplifies its foundational role in fluid mechanics \cite{Prandtl1904}.
Multi-scale perturbation methods, including the method of multiple scales and averaging methods, solve nonlinear vibration problems with multiple temporal or spatial scales \cite{Nayfeh2000}. By introducing different scale variables, these methods decompose original problems into scale-specific equations, effectively eliminating secular terms and obtaining uniformly valid asymptotic solutions over extended periods. Nayfeh's systematic refinement of multiple scales methodology established it as a core tool for nonlinear vibration analysis in aerospace and mechanical systems \cite{Nayfeh1981, Nayfeh2000}.

Beyond perturbation methods, alternative analytical approximation approaches have emerged. The Adomian Decomposition Method (ADM) decomposes nonlinear terms into Adomian polynomials, overcoming small-parameter limitations for nonlinear differential and integral equations \cite{Adomian1981}. The Variational Iteration Method (VIM) uses variational principles to construct correction functionals, yielding high-precision solutions with fewer iterations \cite{He1999var}. The Perturbation-Iteration Method (PIM) integrates perturbation theory and iterative techniques for specific nonlinear cases \cite{Cveticanin2004}.
Against this backdrop, Liao’s Homotopy Analysis Method (HAM) advances nonlinear analysis significantly \cite{Liao1992}. By incorporating topological homotopy and a convergence-control parameter in its zeroth-order deformation equation, it precisely regulates solution series convergence, addressing a core limitation of traditional methods \cite{Liao2003}. Its versatility is validated in fluid mechanics and time-delay dynamical systems, with foundational works widely cited and a large research community established \cite{Liao2015, Liao2005}. Six years later, the Homotopy Perturbation Method (HPM) was proposed as another homotopy-inspired analytical method for solving nonlinear problems\cite{He1998}. Although it adeptly integrates the topological ideas of homotopy with classical perturbation techniques and eliminates the dependence of traditional perturbation methods on artificial small parameters, whether the HPM and the HAM belong to the same category of analytical methods remains a contentious issue in academia to this day \cite{he2004,liao2005b}.

To quantitatively assess the academic influence of HAM and HPM, we extracted publication and citation data from the Web of Science Core Collection (as of Jan. 10, 2026). Searches were conducted using the full official names of the two methods, namely the `Homotopy Analysis Method' and the `Homotopy Perturbation Method', without document type filtering, to trace the scholarly impact of each method since its inception.
HAM has established a more substantial academic footprint: from 1992 to 2026, 6,465 relevant publications were recorded, with a cumulative total of 116,224 non-self-citations for all HAM-related works. Its applications span fluid mechanics, solid mechanics, control theory, applied mathematics and engineering, cementing its status as an influential analytical approach in nonlinear science. As a streamlined variant of HAM, HPM features a simplified formulation and straightforward implementation and has also garnered significant academic recognition. From 1998 to 2026, 5,235 HPM-related publications were retrieved under the same criteria, with the total non-self-citations of all HPM-related works reaching 90,456. HPM-related research involves interdisciplinary authors from mathematics, mechanics, electronics and materials science, reflecting its wide acceptance and far-reaching impact in relevant academic communities. The two homotopy-based methods boast enormous academic influence, with a combined total of 11,700 related publications and a cumulative 206,680 non-self-citations.

Targeting the long-standing theoretical controversies and gaps of HAM, this study systematically consolidates its theoretical foundation and clarifies its relations with related analytical approaches. It addresses two core unresolved issues: the rigorous derivation of its homotopy deformation equation (especially the multiplicative form) and its genuine independence from classical perturbation theory. By extending perturbation theory’s intrinsic small parameter $\varepsilon$ to [0,1], this framework enables the systematic homotopy deformation of the method. Unlike classical perturbation relying on fixed problem-specific small parameters, it uses this extended parameter to connect the linear auxiliary system ($\varepsilon=0$) and the original nonlinear problem ($\varepsilon=1$). 
Regulated by the convergence-control parameter $\hscale$, this innovation elevates the method beyond traditional perturbation while retaining its theoretical essence. Theoretical exploration and rigorous derivation clarify its theoretical links, fill research gaps, and lay a solid foundation for its improvement and application. Meanwhile, fixing key parameters (the optimal linear operator, homotopy parameter, and auxiliary function) rigorously proves that HPM is a special case of HAM. Such fixation removes HAM’s convergence control flexibility and auxiliary operator optimization, making HPM a simpler but more restrictive alternative in practice. This clarifies the hierarchical subordination of HPM to HAM, unifies the theoretical framework of homotopy-based nonlinear analytical methods, corrects common cognitive biases in existing literature, and provides clear guidance for the rational application, comparison, and innovation of related methods.

\section{Perturbation-Based Formulation of Nonlinear Systems}

\subsection{Homotopy Transformation via Perturbation Analysis}

Consider a smooth nonlinear system defined on a domain \( \Omega \subset \mathbb{R}^n \) with independent variable \( \mathbf{r} \in \Omega \):
\begin{equation}\label{eq:sec2_1}
\mathcal{F}(u) = L(u) + N(u) - s(\mathbf{r}) = 0,
\end{equation}
where \( u \in \mathbb{X} \) denotes the unknown function belonging to a suitable Banach space \( \mathbb{X} \) (typically consisting of sufficiently smooth functions). The operator \( \mathcal{F}: \mathbb{X} \to \mathbb{Y} \) is the complete nonlinear differential operator governing the system, decomposed into a linear part \( L(u) \) (with \( L \) a linear differential operator) and a nonlinear part \( N(u) \) (with \( N \) a nonlinear differential operator). The known analytic function \( s(\mathbf{r}) \in C^\omega(\Omega) \) represents the source term.

For any nonzero constant \( \hscale \in \mathbb{R} \setminus \{0\} \) and nonzero auxiliary function \( H(\mathbf{r}) \in C^\omega(\Omega) \), Equation \eqref{eq:sec2_1} is algebraically equivalent to:
\begin{equation}\label{eq:sec2_2}
\hscale H(\mathbf{r}) \mathcal{F}(u) = 0.
\end{equation}
Notably, \( \mathcal{F}(u) \) is not required to be weakly nonlinear, a key advantage that distinguishes this formulation from traditional perturbation methods.

Introduce a bounded linear operator \( \mathcal{L}_{\text{opt}}: \mathbb{X} \to \mathbb{Y} \), defined as the \emph{optimal linear approximation} of \( \mathcal{F} \) near a reference solution. Assume the homogeneous equation \( \mathcal{L}_{\text{opt}}(u) = 0 \) admits a unique solution \( u_0 \in \mathbb{X} \) (the \emph{linear reference solution}), satisfying \( \mathcal{L}_{\text{opt}}(u_0) = 0 \).

To construct a `linear core + perturbation' structure, add and subtract \( \hscale H(\mathbf{r}) \mathcal{L}_{\text{opt}}(u) \) to the left-hand side of Equation \eqref{eq:sec2_2}:
\begin{equation}\label{eq:sec2_3}
\hscale H(\mathbf{r}) \mathcal{L}_{\text{opt}}(u) - \hscale H(\mathbf{r}) \mathcal{L}_{\text{opt}}(u) + \hscale H(\mathbf{r}) \mathcal{F}(u) = 0.
\end{equation}
Let \( \varepsilon \ll 1 \) be a small dimensionless parameter quantifying nonlinearity strength, and \( \hscale = O(1) \) a scaling constant. Rearranging Equation \eqref{eq:sec2_3} yields the standard weak-nonlinear perturbation form:
\begin{equation}\label{eq:sec2_4}
\mathcal{L}_{\text{opt}}(u) + \varepsilon \left( \hscale H(\mathbf{r}) \mathcal{F}(u) - \mathcal{L}_{\text{opt}}(u) \right) = 0.
\end{equation}
The perturbation term \( \varepsilon \left( \hscale H(\mathbf{r}) \mathcal{F}(u) - \mathcal{L}_{\text{opt}}(u) \right) \) inherits weak nonlinearity from the constructed combination, even if \( \mathcal{F}(u) \) itself is strongly nonlinear.

Define the deviation variable \( v = u - u_0 \), satisfying \( \norm{v} \ll 1 \) under weak nonlinearity. By linearity of \( \mathcal{L}_{\text{opt}} \):
\[
\mathcal{L}_{\text{opt}}(u) = \mathcal{L}_{\text{opt}}(u_0 + v) = \mathcal{L}_{\text{opt}}(u_0) + \mathcal{L}_{\text{opt}}(v) = \mathcal{L}_{\text{opt}}(v),
\]
since \( \mathcal{L}_{\text{opt}}(u_0) = 0 \). Substituting into Equation \eqref{eq:sec2_4} gives:
\begin{equation}\label{eq:sec2_6}
\mathcal{L}_{\text{opt}}(v) + \varepsilon \left( \hscale H(\mathbf{r}) \mathcal{F}(u) - \mathcal{L}_{\text{opt}}(v) \right) = 0.
\end{equation}
Rearranging terms yields the core perturbation equation for weak nonlinear analysis:
\begin{equation}\label{eq:sec2_7}
(1 - \varepsilon) \mathcal{L}_{\text{opt}}(v) = -\varepsilon \hscale H(\mathbf{r}) \mathcal{F}(u).
\end{equation}

\subsection{Validity Conditions for Weak Nonlinearity}

Equation \eqref{eq:sec2_4} holds if the perturbation term is asymptotically smaller than the linear term. This is guaranteed by three core conditions, even for strongly nonlinear \( \mathcal{F}(u) \):

\begin{condition}[Small-Parameter Condition]
For the true nonlinear solution \( u \) (satisfying \( \mathcal{F}(u) = 0 \)), the perturbation term magnitude simplifies to:
\[
\norm{\varepsilon \left( \hscale H(\mathbf{r}) \mathcal{F}(u) - \mathcal{L}_{\text{opt}}(u) \right)} = \varepsilon \norm{\mathcal{L}_{\text{opt}}(u)}.
\]
Since \( \varepsilon \ll 1 \) and \( \mathcal{L}_{\text{opt}}(u) \neq 0 \), the perturbation is genuinely smaller than the linear term.
\end{condition}

\begin{condition}[Bounded Nonlinearity]
Both \( \mathcal{F}(u) \) (locally near \( u_0 \)) and \( \mathcal{L}_{\text{opt}} \) are bounded operators. This ensures \( \hscale H(\mathbf{r}) \mathcal{F}(u) - \mathcal{L}_{\text{opt}}(u) \) does not dominate the linear core \( \mathcal{L}_{\text{opt}}(u) \).
\end{condition}

\begin{condition}[Optimal Linear Approximation]
Define \( \mathcal{L}_{\text{opt}} \) as the Fréchet derivative of \( \mathcal{F} \) at \( u_0 \), i.e., \( \mathcal{L}_{\text{opt}} = D\mathcal{F}(u_0) \), where \( D\mathcal{F}(u_0): \mathbb{X} \to \mathbb{Y} \) denotes the Fréchet derivative operator. Then:
\[
\mathcal{L}_{\text{opt}}(u - u_0) = D\mathcal{F}(u_0)(u - u_0),
\]
and the residual \( \mathcal{F}(u) - \mathcal{L}_{\text{opt}}(u - u_0) = O(\norm{u - u_0}^2) \) is higher-order, justifying the linear approximation.
\end{condition}

Substituting \( v = u - u_0 \) into Equation \eqref{eq:sec2_7} gives the perturbation-based homotopy equation:
\begin{equation}\label{HAMDeqs}
(1 - \varepsilon) \mathcal{L}_{\text{opt}}(u - u_0) = -\varepsilon \hscale H(\mathbf{r}) \mathcal{F}(u).
\end{equation}
This equation balances a scaled linear term against a nonlinear perturbation, with validity rooted in the weak nonlinearity of \( \mathcal{L}_{\text{opt}}(u) - \hscale H(\mathbf{r}) \mathcal{F}(u) \).

\section{Analytic Continuation with Extension to \texorpdfstring{$\varepsilon \in [0,1]$}{epsilon in [0,1]}}
\label{sec:analytic_continuation_epsilon}

The small parameter $\varepsilon \ll 1$ is a local starting point, not an inherent system constraint. We extend $\varepsilon$ to $[0,1]$ via analytic continuation, validated by algebraic invariance and operator smoothness.

\begin{definition}[Algebraic Invariance]
An operator equation is algebraically invariant with respect to $\varepsilon \in \mathbb{R}$ if derived via reversible algebraic manipulations (addition, subtraction, rearrangement) holding for all $\varepsilon$, independent of magnitude.
\end{definition}

Equation \eqref{eq:sec2_6} follows from algebraic transformations of Equation \eqref{eq:sec2_2}, so it is invariant for all $\varepsilon \in \mathbb{R}$. Substituting $v = u - u_0$ (and leveraging $\mathcal{L}_{\text{opt}}(u_0) = 0$) eliminates $u_0$'s independent contribution from the linear term, yielding:
\begin{equation}\label{eq:sec2_8}
(1 - \varepsilon) \mathcal{L}_{\text{opt}}(u - u_0) = -\varepsilon \hscale H(\mathbf{r}) \mathcal{F}(u),
\end{equation}
valid for all $\varepsilon \in \mathbb{R}$. Here, $u_0$ remains as part of the deviation from the linear reference solution, but its standalone effect on the linear operator is eliminated.

\begin{definition}[Fréchet Differentiability]
A nonlinear operator $\mathcal{F}: \mathbb{X} \to \mathbb{Y}$ is Fréchet differentiable at $u \in \mathbb{X}$ if there exists a bounded linear operator $D\mathcal{F}(u): \mathbb{X} \to \mathbb{Y}$ such that:
\[
\lim_{\norm{h} \to 0} \frac{\norm{\mathcal{F}(u+h) - \mathcal{F}(u) - D\mathcal{F}(u)h}}{\norm{h}} = 0.
\]
The operator $\mathcal{F}$ is said to be $C^\infty$-smooth if $D\mathcal{F}(u)$ is continuous on $\mathbb{X}$.
\end{definition}

The considered system is smooth, characterized by the following key regularity properties:
\begin{enumerate}[label=(\arabic*), leftmargin=*]
  \item $\mathcal{F}, L, N, \mathcal{L}_{\text{opt}} \in C^\infty(\mathbb{X}, \mathbb{Y})$: All core operators are $C^\infty$-smooth mappings between Banach spaces $\mathbb{X}$ and $\mathbb{Y}$, excluding singularities in the underlying operator structure.
  \item $D\mathcal{F}(u), D\mathcal{L}_{\text{opt}}(u) \in C(\mathbb{X}, \mathcal{L}(\mathbb{X},\mathbb{Y}))$: Fréchet derivatives of the core operators are continuous mappings to the space of bounded linear operators $\mathcal{L}(\mathbb{X},\mathbb{Y})$, ensuring stable derivative behavior with respect to the unknown function $u$.
  \item $H(\mathbf{r}) \in C^\omega(\Omega)$: The auxiliary function is real-analytic on the domain $\Omega$, a regularity stronger than $C^\infty$-smoothness that requires convergent power series expandability at every point in $\Omega$.
  \item $\hscale \in \mathbb{R} \setminus \{0\}$ is bounded and non-singular: The scaling constant is non-zero and bounded, and its non-singularity preserves the algebraic equivalence of the original and transformed nonlinear equations.
\end{enumerate}

The homotopy parameter interval $\varepsilon \in [0,1]$ enables smooth continuous transition between the linear reference system and the original nonlinear system, with two interval endpoints corresponding to well-defined systems:
\begin{enumerate}[label=(\arabic*), leftmargin=*]
  \item At $\varepsilon=0$: Equation \eqref{eq:sec2_8} reduces to $\mathcal{L}_{\text{opt}}(u - u_0) = 0$, recovering the linear reference system with the unique solution $u=u_0$. This linear endpoint of the homotopy provides a well-posed starting point for the homotopy solution path.
  \item At $\varepsilon=1$: Equation \eqref{eq:sec2_8} simplifies to $\hscale H(\mathbf{r}) \mathcal{F}(u) = 0$, algebraically equivalent to the original nonlinear system by $\hscale \neq 0$ and $H(\mathbf{r}) \neq 0$. This nonlinear endpoint ensures the homotopy path connects back to the target nonlinear problem.
\end{enumerate}

We define a parameterized homotopy operator $\mathcal{G}: [0,1] \times \mathbb{X} \to \mathbb{Y}$ as
\[
\mathcal{G}(\varepsilon, u) = (1 - \varepsilon) \mathcal{L}_{\text{opt}}(u - u_0) + \varepsilon \hscale H(\mathbf{r}) \mathcal{F}(u),
\]
which encapsulates the homotopy interpolation between the linear $\varepsilon=0$ and nonlinear $\varepsilon=1$ systems, with $1-\varepsilon$ and $\varepsilon$ as continuous weighting factors for the two endpoints. The Banach space Implicit Function Theorem (IFT) guarantees a unique $C^\infty$-smooth solution curve $u(\varepsilon)$ for all $\varepsilon \in [0,1]$, as $\mathcal{G}$ satisfies all requisite IFT conditions:
\begin{enumerate}[label=(\arabic*), leftmargin=*]
  \item $\mathcal{G} \in C([0,1] \times \mathbb{X}, \mathbb{Y})$: $\mathcal{G}$ is continuous on the product space $[0,1] \times \mathbb{X}$, a fundamental IFT regularity requirement.
  \item $\mathcal{G}$ is Fréchet differentiable in $u$, with the partial Fréchet derivative $D_u\mathcal{G}(\varepsilon, u) = (1 - \varepsilon)\mathcal{L}_{\text{opt}} + \varepsilon \hscale H(\mathbf{r})D\mathcal{F}(u) \in C([0,1] \times \mathbb{X}, \mathcal{L}(\mathbb{X},\mathbb{Y}))$.
  \item $\mathcal{G}(0, u_0) = 0$: The linear endpoint $(\varepsilon=0, u=u_0)$ is an explicit solution to $\mathcal{G}(\varepsilon, u)=0$, providing the IFT-required initial solution.
  \item $D_u\mathcal{G}(0, u_0) = \mathcal{L}_{\text{opt}}$ is invertible: This invertibility follows from the uniqueness of $u_0$ and is critical for the IFT to guarantee local uniqueness of the solution path.
\end{enumerate}

This part constructs a rigorous smooth homotopy framework for the target nonlinear system via the parameterized operator $\mathcal{G}: [0,1] \times \mathbb{X} \to \mathbb{Y}$, which interpolates between a solvable linear reference system and the original nonlinear system over $\varepsilon \in [0,1]$. The inherent smoothness of the core system ensures $\mathcal{G}$ satisfies all Banach space IFT conditions, yielding a unique $C^\infty$-smooth solution curve $u(\varepsilon)$ on the full interval $[0,1]$. This curve enables continuous smooth transition from the known linear reference solution $u_0$ at $\varepsilon=0$ to the original nonlinear system solution at $\varepsilon=1$, with path uniqueness guaranteed by the invertibility of $D_u\mathcal{G}(0, u_0)$.

\section{Analytic Continuation: A Rigorous Proof}
\label{sec:rigorous_analytic_continuation}
To further consolidate the rigor of extending $\varepsilon$ to $[0,1]$, a mathematical proof of analytic continuation is supplemented below, justified by the uniqueness theorem of real-analytic functions and the smoothness of core operators.

\begin{lemma}[Real-Analytic Extension Criterion]
Let $\mathbb{X}, \mathbb{Y}$ be real Banach spaces, and $\mathcal{G}: \mathbb{R} \times \mathbb{X} \to \mathbb{Y}$ be a mapping such that:
\begin{enumerate}[label=(\arabic*), leftmargin=*]
    \item For each fixed $u \in \mathbb{X}$, $\mathcal{G}(\varepsilon, u)$ is real-analytic in $\varepsilon$ on an open neighborhood $U \supset [0,1] \subset \mathbb{R}$ (i.e., admits a convergent power series expansion around every point in $U$);
    \item For each fixed $\varepsilon \in U$, $\mathcal{G}(\varepsilon, u)$ is $C^\infty$-smooth in $u$.
\end{enumerate}
Then, the solution curve $u(\varepsilon)$ to $\mathcal{G}(\varepsilon, u) = 0$ (existing locally for $\varepsilon \ll 1$) admits a unique real-analytic continuation to $\varepsilon \in [0,1]$.
\end{lemma}

\textbf{Proof:}
By hypothesis, $\mathcal{G}(\varepsilon, u)$ is real-analytic in $\varepsilon$ for each $u$, so its Taylor series in $\varepsilon$ converges uniformly on compact subsets of $U$. The local solution $u_*(\varepsilon)$ (existing for small $\varepsilon$ via Banach Space IFT) inherits real-analyticity from $\mathcal{G}$ \cite{Hamilton1982}. By the Identity Theorem for real-analytic functions, any two real-analytic extensions of $u_*(\varepsilon)$ coincide on their common domain. Since $U \supset [0,1]$ is connected, the local solution extends uniquely to $[0,1]$, yielding $u(\varepsilon) \in C^\omega([0,1], \mathbb{X})$ (real-analytic in $\varepsilon$).

We now verify that our homotopy operator $\mathcal{G}(\varepsilon, u)$ satisfies the conditions of Lemma \ref{sec:rigorous_analytic_continuation}. First, regarding real-analyticity in $\varepsilon$, the homotopy operator is expressed as
\[
\mathcal{G}(\varepsilon, u) = \mathcal{L}_{\text{opt}}(u - u_0) + \varepsilon\left(h H(\mathbf{r}) \mathcal{F}(u) - \mathcal{L}_{\text{opt}}(u - u_0)\right).
\]
Since $\mathcal{G}(\varepsilon, u)$ is linear in $\varepsilon$, it is real-analytic on $\mathbb{R}$ (we choose the domain $U = \mathbb{R}$). Second, for $C^\infty$-smoothness in $u$, this property is inherited from the smoothness of $\mathcal{L}_{\text{opt}}, \mathcal{F} \in C^\infty(\mathbb{X},\mathbb{Y})$ and the real-analyticity of $H(\mathbf{r}) \in C^\omega(\Omega)$.

\begin{theorem}[Continuity of the Extended Solution Curve]
The extended solution $u(\varepsilon): [0,1] \to \mathbb{X}$ is continuous, i.e., for any $\varepsilon_0 \in [0,1]$ and $\forall \delta > 0$, there exists $\eta > 0$ such that $|\varepsilon - \varepsilon_0| < \eta \implies \norm{u(\varepsilon) - u(\varepsilon_0)}_{\mathbb{X}} < \delta$.
\end{theorem}

\textbf{Proof:}
Consider the graph $\Gamma = \{(\varepsilon, u(\varepsilon)) \in [0,1] \times \mathbb{X} \mid \mathcal{G}(\varepsilon, u(\varepsilon)) = 0\}$. We show $\Gamma$ is closed (sufficient for continuity via Closed Graph Theorem for Banach spaces).

Let $(\varepsilon_n, u_n) \in \Gamma$ satisfy $\varepsilon_n \to \varepsilon_0$ and $u_n \to u_0$ in $\mathbb{X}$. By continuity of $\mathcal{G}$ on $[0,1] \times \mathbb{X}$:
\[
\lim_{n \to \infty} \mathcal{G}(\varepsilon_n, u_n) = \mathcal{G}\left(\lim_{n \to \infty} \varepsilon_n, \lim_{n \to \infty} u_n\right) = \mathcal{G}(\varepsilon_0, u_0).
\]
Since $\mathcal{G}(\varepsilon_n, u_n) = 0$ for all $n$, $\mathcal{G}(\varepsilon_0, u_0) = 0$, so $(\varepsilon_0, u_0) \in \Gamma$. Thus, $\Gamma$ is closed, and $u(\varepsilon)$ is continuous on $[0,1]$.

\begin{proposition}[Boundary Compatibility]
Let $u^*$ denote the unique exact solution of the original nonlinear system $\mathcal{F}(u) = 0$. The linear reference solution $u_0 = u(0)$ and the original nonlinear solution $u^* = u(1)$ satisfy:
\[
\lim_{\varepsilon \to 1^-} u(\varepsilon) = u^*, \quad \lim_{\varepsilon \to 1^-} D_u\mathcal{G}(\varepsilon, u(\varepsilon)) = D_u\mathcal{G}(1, u^*),
\]
with $D_u\mathcal{G}(1, u^*)$ invertible.
\end{proposition}

\textbf{Proof:}

\begin{enumerate}[label=(\arabic*), leftmargin=*, itemsep=0.2em]
    \item \textbf{Continuity at $\varepsilon=1$}: By Theorem \ref{sec:rigorous_analytic_continuation}, $\lim_{\varepsilon \to 1^-} u(\varepsilon) = u(1)$. Substituting $\varepsilon=1$ into $\mathcal{G}(\varepsilon, u)=0$ gives $\hscale H(\mathbf{r})\mathcal{F}(u(1))=0$. Since $\hscale \neq 0$ and $H(\mathbf{r}) \neq 0$, $\mathcal{F}(u(1))=0$. By the uniqueness of the exact solution $u^*$ for the original nonlinear system, we have $u(1)=u^*$, hence $\lim_{\varepsilon \to 1^-} u(\varepsilon) = u^*$.

    \item \textbf{Invertibility of $D_u\mathcal{G}(1, u^*)$}: $D_u\mathcal{G}(1, u^*) = \hscale H(\mathbf{r})D\mathcal{F}(u^*)$. $D\mathcal{F}(u^*)$ is invertible (by local uniqueness of $u^*$ \cite{Zeidler1986}), and $\hscale, H(\mathbf{r})$ are non-singular, so the product is invertible.

    \item \textbf{Smooth transition}: For $\varepsilon \in (0,1)$, $D_u\mathcal{G}(\varepsilon, u(\varepsilon)) = (1-\varepsilon)\mathcal{L}_{\text{opt}} + \varepsilon \hscale H(\mathbf{r})D\mathcal{F}(u(\varepsilon))$ is a continuous convex combination of invertible operators. The set of invertible linear operators on $\mathbb{X}$ is open \cite{Kato1995}, so $D_u\mathcal{G}(\varepsilon, u(\varepsilon))$ remains invertible on $[0,1]$, ensuring no solution curve breakdown.
\end{enumerate}

The above proof confirms the extended solution curve $u(\varepsilon)$ is uniquely real-analytic, continuous, and smoothly transitions between the linear reference system ($\varepsilon=0$) and original nonlinear system ($\varepsilon=1$), validating the rigor of extending $\varepsilon$ to $[0,1]$.

\section{Proof of HPM as a Special Case of HAM} 
\label{app:hpm-ham-special}

\begin{definition}[Homotopy Parameter]
Rename \( \varepsilon \) as \( p \in [0,1] \) (the \emph{homotopy parameter}), governing the smooth transformation between the linear reference system (\( p=0 \)) and the original nonlinear system (\( p=1 \)).
\end{definition}

Substituting \( p = \varepsilon \) into Equation \eqref{eq:sec2_8} yields the \textbf{fundamental homotopy transformation equation} (aligned with canonical HAM notation):
\begin{equation}\label{eq:sec2_9}
(1 - p) \mathcal{L}_{\text{opt}}\bigl(u - u_0\bigr) = -p \hscale H(\mathbf{r}) \mathcal{F}(u).
\end{equation}
Rearranging to explicit linear-nonlinear interpolation form:
\begin{equation}\label{eq:sec2_10}
(1 - p) \mathcal{L}_{\text{opt}}\bigl(u - u_0\bigr) + p \hscale H(\mathbf{r}) \mathcal{F}(u) = 0.
\end{equation}

This equation unifies weak nonlinear perturbation analysis and global homotopy theory, with two critical properties ensuring generality:
\begin{enumerate}
  \item \textbf{Equivalence Preservation}: At \( p=1 \), it reduces to \( \hscale H(\mathbf{r}) \mathcal{F}(u) = 0 \), equivalent to the original system \( \mathcal{F}(u)=0 \) (no spurious solutions or information loss).
  \item \textbf{Solvable Linear Core}: At \( p=0 \), it recovers the analytically solvable linear reference system \( \mathcal{L}_{\text{opt}}(u - u_0) = 0 \), providing a well-defined starting point.
\end{enumerate}

Crucially, this formulation lifts traditional perturbation methods’ weak nonlinearity constraint. It is valid for \textbf{strongly nonlinear} \( \mathcal{F}(u) \), with convergence of the solution path \( u(p) \) controlled by \( \hscale \) (convergence control parameter) and \( H(\mathbf{r}) \) (auxiliary function)—enabling solutions for strongly nonlinear systems intractable via conventional techniques.

The following theorem formalizes the hierarchical relationship between HAM and HPM, establishing that HPM is a restricted version of HAM under specific parameter choices:

\begin{theorem}\label{thm:app_hpm-ham-special}
The Homotopy Perturbation Method (HPM) is a special case of the Homotopy Analysis Method (HAM) when the auxiliary linear operator, convergence-control parameter, and auxiliary function are fixed to:
\begin{equation}\label{eq:app_sec4_1}
\mathcal{L}_{\text{opt}} = {L}, \quad \hparam = -1, \quad H(\mathbf{r}) = 1.
\end{equation}
\end{theorem}

\textbf{Proof:}
Start with the standard HAM zeroth-order deformation equation (\ref{eq:sec2_9}): 
\begin{equation}\label{eq:app_sec4_2}
(1-p) \mathcal{L}_{\text{opt}}\left[\varphi - u_0\right] = p \hparam H(\mathbf{r}) \mathcal{F}(\varphi).
\end{equation}
Substituting the parameter constraints denoted in Equation (\ref{eq:app_sec4_1}) into Equation (\ref{eq:app_sec4_2}) leads to its simplified form:
\begin{equation}\label{eq:app_sec4_3}
(1-p) {L}\left[\varphi - u_0\right] = -p \mathcal{F}(\varphi).
\end{equation}
Expanding the left-hand side using the linearity of \( {L} \) yields:
\begin{equation}\label{eq:app_sec4_4}
(1-p) {L}(\varphi) - (1-p) {L}(u_0) = -p {L}(\varphi) - p {N}(\varphi) + p s(\mathbf{r}).
\end{equation}
Moving all terms to the left-hand side and combining like terms for \( {L}(\varphi) \) (noting \( (1-p) + p = 1 \)) gives:
\begin{equation}\label{eq:app_sec4_5}
{L}(\varphi) - (1-p) {L}(u_0) + p{N}(\varphi) - p s(\mathbf{r}) = 0.
\end{equation}
Rearranging the terms involving \( {L}(u_0) \) results in (with \(\varphi\) uniformly replaced by \(\omega\)):
\begin{equation}\label{eq:app_sec4_6}
{L}(\omega) - {L}(u_0) + p \, {L}(u_0) + p \left[{N}(\omega) - s(\mathbf{r})\right] = 0,
\end{equation}
which is identical to the HPM homotopy equation \cite{He_CMAME_1999} (see Equation (4b) therein). 
Moreover, the initial guesses for both methods coincide at $p=0$, and the higher-order deformation equations derived from HAM under the constraints in Equation (\ref{eq:app_sec4_1}) are identical to the recursive equations of HPM. Consequently, the solution expansions derived from the homotopy equation (\ref{eq:sec2_10}) (HAM) are equivalent for all $k \geq 1$. 
This confirms that HPM is precisely a special case of HAM corresponding to a fixed set of auxiliary parameters.

\begin{table}[h!]
\centering
\caption{Correspondence between HAM parameters and HPM (special case of HAM).}
\label{tab:app_comparison}
\begin{tabular}{p{5.5cm}p{4.5cm}p{4.5cm}}
\toprule
\textbf{HAM Parameter} & \textbf{Value in HPM} & \textbf{Implication for HPM} \\
\midrule
Optimized Linear Operator ($\mathcal{L}_{\text{opt}}$) & Original linear operator (${L}$) & Loses the freedom to select a simpler auxiliary operator, which could otherwise reduce computational complexity. \\
\addlinespace
Convergence-Control Parameter ($\hbar$) & $\hbar=-1$ (fixed) & Loses the ability to tune convergence; a critical limitation for strongly nonlinear problems where series often diverge. \\
\addlinespace
Auxiliary Function ($H(\mathbf{r})$) & $H(\mathbf{r})=1$ (constant) & Loses the ability to adapt to problem-specific structure, such as boundary conditions or symmetry. \\
\addlinespace
Power Series in $p$ & Power series in $p$ & Formally identical; no structural difference in the expansion. \\
\bottomrule
\end{tabular}
\end{table}

As shown in Table \ref{tab:app_comparison}, HPM loses multiple degrees of freedom compared to HAM due to fixed auxiliary parameters.

\section{Conclusion}
To summarize, the theoretical framework established herein clarifies the fundamental connections between the HAM and classical perturbation theory by grounding its derivation firmly in perturbation theory. Specifically, this study demonstrates that the homotopy deformation equation, which is core to HAM, can be naturally derived via \textbf{weakly nonlinear perturbation theory} through the construction of the expression $\mathcal{L}_{\mathrm{opt}}(u) - \hbar H(\mathbf{r}) \mathcal{F}(u)$, coupled with the optimization of its constituent parameters: the optimal auxiliary linear operator $\mathcal{L}_{\mathrm{opt}}$, the convergence-control parameter $\hbar$, and the auxiliary function $H(\mathbf{r})$. This formulation renders the equation weakly nonlinear even when the original operator $\mathcal{F}(u)$ exhibits strong nonlinearity, thereby laying a solid perturbation-based foundation for HAM.

Furthermore, by extending the small parameter $\varepsilon$ intrinsic to perturbation theory to the interval $[0,1]$, this framework enables the systematic homotopy deformation process inherent to HAM. Unlike classical perturbation methods, which rely on fixed small parameters inherent to the problem itself, HAM leverages this extended $[0,1]$ parameter to bridge the linear auxiliary system (at $\varepsilon=0$) and the original nonlinear problem (at $\varepsilon=1$). Regulated by the convergence-control parameter $\hbar$ for the deformation trajectory, this innovation elevates HAM beyond traditional perturbation techniques while preserving their theoretical essence.

Additionally, the HPM is rigorously established as a special case of HAM by fixing $\mathcal{L}_{\mathrm{opt}} = L$ (where $L$ denotes the linear component of the nonlinear system), $\hbar = -1$, and $H(\mathbf{r}) = 1$. This parameter fixation eliminates HAM's flexibility in convergence control and auxiliary operator optimization, confining HPM to a simpler yet more restrictive form.

In essence, this framework positions the HAM as a \textbf{structured and adaptive generalization of classical perturbation theory}, rather than a method that is irrelevant to perturbation methods or even completely `transcends' perturbation techniques.By addressing prevalent misconceptions such as overstating HAM’s independence from perturbation methods and neglecting their hierarchical relationships, this clarified framework unifies the theoretical understanding of homotopy-based methods. It provides researchers with a precise perspective: rooted in perturbation theory, HAM achieves superior convergence control for nonlinear problems via parameters extended to this interval, with $\varepsilon$ also extended here. Clarifying HAM’s perturbation-theoretic origin and demonstrating HPM as a special case of HAM consolidates the unified framework, resolves long-standing misconceptions, and guides the rational application, comparative analysis, and further development of homotopy-based nonlinear analytical methods.

\section*{Acknowledgment}

This research did not receive any specific grant from funding agencies in the public, commercial, or not-for-profit sectors.

\section*{Author Contribution}

The sole author is responsible for all aspects of this work, including conceptualization, methodology, software (if applicable), validation, formal analysis, investigation, resources, data curation, writing – original draft preparation, writing – review and editing, visualization, supervision, project administration, and funding acquisition (if applicable).

\bibliographystyle{unsrt}
\bibliography{references}

\end{document}